\documentclass[reqno,english]{amsart}
\usepackage{amsfonts,amsmath,latexsym,verbatim,amscd,mathrsfs,color,array}
\usepackage[colorlinks=true]{hyperref}
\usepackage{amsthm,amsfonts}
\usepackage{amssymb}
\usepackage{pdfsync}
\usepackage{epstopdf}
\usepackage{cite}
\usepackage{graphicx}
\usepackage{datetime}

\usepackage{booktabs}

\usepackage{caption2}
\usepackage{subfigure}
\usepackage{float}
\usepackage{wrapfig}

\newtheorem{theorem}{Theorem}[section]

\newtheorem{cor}{Corollary}[section]
\newtheorem{prop}{Proposition}[section]
\newtheorem{remark}{Remark}[section]

\numberwithin{equation}{section}

\begin{document}
\title[Hamilton-Jacobi-Bellman equations]{Deep neural network approximations for the stable manifolds of the Hamilton-Jacobi-Bellman equations}
\author[G. Chen]{Guoyuan Chen}
\address{\noindent
School of Data Sciences, Zhejiang University of Finance \& Economics, Hangzhou 310018, Zhejiang, P. R. China}
\email{gychen@zufe.edu.cn}

\begin{abstract}
For an infinite-horizon control problem, the optimal control can be represented by the stable manifold of the characteristic Hamiltonian system of Hamilton-Jacobi-Bellman (HJB) equation in a semiglobal domain. In this paper, we first theoretically prove that if an approximation is sufficiently close to the exact stable manifold of the HJB equation in a certain sense, then the control derived from this approximation stabilizes the system and is nearly optimal. Then, based on the theoretical result, we propose a deep learning algorithm to approximate the stable manifold and compute optimal feedback control numerically. The algorithm relies on adaptive data generation through finding trajectories randomly within the stable manifold.
Such kind of algorithm is grid-free basically, making it potentially applicable to a wide range of high-dimensional nonlinear systems. We demonstrate the effectiveness of our method through two examples: stabilizing the Reaction Wheel Pendulums and controlling the parabolic Allen-Cahn equation.
\end{abstract}
\maketitle


\section{Introduction}

It is well known that solving Hamilton-Jacobi-Bellman (HJB) equation is essentially complicated in general and requires advanced computation tools and techniques.
In this paper, we focus on the stationary HJB equations which are related to infinite-horizon optimal control and $H^\infty$ control problems. The stabilizing solutions of the stationary HJB equations correspond to the stable manifolds of the characteristic Hamiltonian systems at equilibriums (cf. e.g. \cite{van1991state, sakamoto2008}). Once the stable manifold is obtained, the optimal control can be represented directly by it, without using the gradient of the solution of the HJB equation (\cite{sakamoto2008}). This approach can be viewed as an extension of LQ theory to nonlinear systems.

Computing stable manifolds for Hamilton-Jacobi (HJ) equations is generally a challenging task. In \cite{sakamoto2008}, the authors introduced an iterative procedure aimed at approximating the precise trajectories of the characteristic Hamiltonian system on the stable manifold. While this method has demonstrated effectiveness in generating effective feedback controls for specific initial conditions or near nominal trajectories, as demonstrated in \cite{sakamoto2013case, horibe2017optimal, horibe2018nonlinear, chen2021symplectic}, it can become time-consuming when applied to computing feedback controls for more general initial states or for high-dimensional state systems.

Numerous studies have been dedicated to numerically solving various HJ equations. See e.g. \cite{cacace2012patchy, darbon2016algorithms, kalise2018polynomial, jiang2016using, beard1997galerkin, mceneaney2007curse, ohtsuka2010solutions}. However, as pointed out in \cite{nakamura2019adaptive}, these existing methods may encounter several limitations, including heavy computational costs for high-dimensional systems, challenges in estimating solution accuracy for general systems, solutions being limited to a small neighborhood of a fixed point or nominal trajectory, and the necessity for special structures in the nonlinear terms of the system.

In recent years, various deep learning methods have emerged for efficiently solving HJ equations in large high-dimensional domains using grid-free sampling. For example, \cite{sirignano2018dgm} suggests utilizing neural networks (NN) to approximate solutions of the HJ equation by minimizing the residual of the PDEs and boundary conditions on a set of randomly sampled collocation points. \cite{han2018solving} and \cite{raissi2018forward} introduce deep learning method for solving PDEs, including HJB equations, by reformulating the PDEs as stochastic differential equations. \cite{kunish2021semiglobal} establishes a theory for employing NN to approximate optimal feedback laws on appropriate functional spaces.
\cite{nakamura2019adaptive} and \cite{kang2019algorithms} propose a causality-free data-driven deep learning algorithm for solving HJB equations, enhancing the effectiveness of NN training through adaptive data generation. Interested readers can refer to \cite{weinan2021algorithms} for a comprehensive review of this direction.

The aim of this paper is to seek semiglobal deep NN approximations of stable manifolds. The main contributions are as follows.
Firstly, we theoretically prove that, given appropriate accuracy assumptions, the control derived from the NN approximation of the stable manifold stabilizes the system and is sufficiently close to the exact optimal control (Theorem \ref{t:decay-xnn} in Section \ref{s:errors}).
Secondly, based on the theoretical results, we propose an algorithm for finding deep NN approximations of stable manifolds. One of the crucial aspects of this algorithm is a composite loss function that incorporates the maximum error, the mean error of the NN from the exact stable manifold on the sample set, and the error between the derivative of the NN at the origin and the stabilizing solution of the Riccati equation (as shown in equation \eqref{e:loss} below).
Another crucial issue is adaptive data generation by solving the characteristic Hamiltonian systems which is inspired by \cite{nakamura2019adaptive}. Specifically, we solve two-point boundary value problems (BVPs) locally near the equilibrium and extend the local solutions using initial value problems (IVPs) for the characteristic Hamiltonian system. We randomly choose a number of samples along each trajectory, and adaptively select additional samples near points with large errors from the previous round of training. Our approach is causality-free and does not depend on discretizing the space, making it suitable for high-dimensional problems. Causality-free algorithms have been successful in various applications. See e.g. \cite{kang2015causality,kang2017solving, kang2017mitigating, yegorov2018perspectives,chow2019algorithm, han2018solving}.

Our approach differs from those focused on solving the HJB equations, e.g., \cite{sirignano2018dgm, nakamura2019adaptive}. Our method is based on the stable manifold, an intrinsic geometric property of the HJB equation. With this framework, we can ensure the stability of the closed loop from the controller generated by the trained NN satisfying certain accuracy. There are few theoretical results on this topic in the literature. In empirical algorithms, the `equilibrium' of the closed loop system from the NN may become unstable or disappear as time goes to infinity, as shown in \cite{nakamura2022neural}. Moreover, our method is different from that in \cite{nakamura2022neural}, which devises certain architectures for approximate NN to stabilize the system. It is worth noting that in \cite{sakamoto2008, sakamoto2013case, chen2021symplectic}, the algorithms are based on an iterative procedure in a small neighborhood of the equilibrium, which is difficult to estimate the accuracy and is time-consuming to generate trajectories.


To demonstrate the effectiveness of our method, we present two applications: stabilizing the Reaction Wheel Pendulum (RWP) and optimal control of the parabolic Allen-Cahn (AC) equation. The NN training for both problems is performed on an ordinary laptop, demonstrating the practicality and efficiency of our approach. In particular, by simulation, the time to generate the control signal from the trained NN takes less than one millisecond on average. This implies that our method enables real-time control in practice.

The paper is structured as follows. Section \ref{s:stable} gives an overview of stable manifold method. Section \ref{s:errors} presents theoretical results. Section \ref{e:algorithm} outlines the deep learning algorithm. Sections \ref{s:pendulum} and \ref{s:AC} give the examples. Appendix \ref{a:proof} includes proof of Theorem \ref{t:decay-xnn}.

\section{The stable manifolds of the HJB equations}\label{s:stable}

In this section, we outline some basic results about the stable manifolds of the HJB equations from nonlinear control theory.
Although the stable manifold method can be applied to problems that are nonlinear in control, to emphasize our main point clearly, we focus on control-affine systems
\begin{eqnarray}\label{e:system}
\dot{x}=f(x)+g(x)u,\quad \mbox{in }\Omega,
\end{eqnarray}
where $g$ is Lipchitz continuous $n\times m$ matrix-valued function, $u$ is an $m$-dimensional feedback control, $\Omega\subset \mathbb R^n$ is a bounded domain with piecewise smooth boundary containing $0$.
Let the instantaneous cost be
$
L(x,u)=q(x)+\frac{1}{2}u^TWu,
$
where $W>0$ is an $m\times m$ symmetric matrix, $q(x)$ is a smooth nonnegative function. Define the cost functional by
$
J(x,u)= \int_0^{+\infty}L(x(t),u(t))dt.
$
The corresponding HJB equation is
\begin{equation}\label{e:HJB}
\nabla V^T(x)f(x)-1/2\nabla V^T(x)R\nabla V(x)+q(x)
=0.
\end{equation}
Here $\nabla V=(\frac{\partial V}{\partial x_1},\cdots, \frac{\partial V}{\partial x_n})^T$ is the gradient of the value function $V$, and
$$
R(x):=g(x)W^{-1}g(x)^T.
$$
The corresponding feedback control function
$
u(x)=-W^{-1}g(x)^T \nabla V(x).
$
Plugging it into \eqref{e:system}, we obtain the closed-loop system
\begin{eqnarray}\label{e:system-closed}
\dot{x}=f(x)-R(x)\nabla V(x),\quad \mbox{in }\Omega.
\end{eqnarray}
Throughout the paper, the following assumption will be made:
\begin{enumerate}
  \item[$(C_1)$] Assume that $f(x)$ and $q(x)$ are $C^{\infty}$ in $\Omega$. For $|x|$ small, $f(x)=Ax+O(|x|^2)$ and $q(x)=\frac{1}{2}x^TQx+O(|x|^3)$ with $A\in\mathbb R^{n\times n}$, and $Q\in \mathbb R^{n\times n}$ is symmetric. Function $g:\Omega\to\mathbb R^{n\times m}$ is $C^{\infty}$.
\end{enumerate}
Recall that a solution $V$ of \eqref{e:HJB} is said to be the \emph{stabilizing solution} if $\nabla V(0)=0$ and $0$ is an asymptotically stable equilibrium of \eqref{e:system-closed}.
The characteristic Hamiltonian system of \eqref{e:HJB} is
\begin{eqnarray}\label{e:Hamiltonian-flow}
\left\{\begin{array}{l}
         \dot{x}= f(x)-R(x)p\\
         \dot{p}= -(\dfrac{\partial f(x)}{\partial x})^Tp+\dfrac{1}{2}\dfrac{\partial (p^TR(x)p)^T}{\partial x}-(\dfrac{\partial q}{\partial x})^T.
       \end{array}
\right.
\end{eqnarray}
Assume $V$ is a stabilizing solution of \eqref{e:HJB}, let
$
\Lambda_V:=\left\{(x,p)\,\left|\,p=\nabla V(x)\right.\right\}.
$
Then
$\Lambda_V$ is invariant under the flow \eqref{e:Hamiltonian-flow} (see e.g. \cite{van1991state}). Note that $0$ is an equilibrium of \eqref{e:Hamiltonian-flow} and the Hamiltonian matrix at $0$ is
$
{\rm Ham}=\left(
                           \begin{array}{cc}
                             A & -R(0) \\
                             -Q & -A^T \\
                           \end{array}
                         \right).\notag
$
We assume the following condition:
\begin{enumerate}
  \item [$(C_2)$] ${\rm Ham} $ is hyperbolic and the generalized eigenspace $E_-$ for $n$-stable eigenvalues satisfies the complementary condition
  $
  E_-\oplus {\rm Im}\left(0, I_n\right)^T=\mathbb R^{2n},
  $ where $I_n$ is the identity matrix of dimension $n$.
\end{enumerate}
Then the following result for stable manifolds holds (\cite{van1991state}).

\begin{theorem}\label{t:stable-manifolds}
Assume that $f,q, R$ satisfy conditions $(C_1-C_2)$. Then the stable manifold of \eqref{e:Hamiltonian-flow} through $(x,p)=(0,0)$ is a smooth submanifold of dimension $n$ in $\mathbb R^{2n}$. Moreover, in a neighborhood of $(0,0)$, this submanifold is the graph $\Lambda_V$. In particular, there exist $\delta>0$ and $k>0$ such that for all $|x|<\delta$,
$
|p(x)-Px|\le k|x|^2,
$
where $P=\frac{\partial^2V}{\partial x^2}(0)$.
\end{theorem}

Recall that $P$ is also the stabilizing solution of the Riccati equation (\cite{van1991state})
\begin{eqnarray}\label{e:riccati}
PA+A^TP-PR(0)P+Q=0,
\end{eqnarray}
which can be considered as quadratic approximation of the HJB equation \eqref{e:HJB} at $0$.
\begin{remark}\label{r:regularity}
From Theorem \ref{t:stable-manifolds}, the stabilizing solution $V$ of \eqref{e:HJB} is $C^{\infty}$ in a neighborhood of $0$ in $\mathbb R^n$. In this paper, we assume that $V$ is a $C^2$ function, and the graph representation of the stable manifold, $\Lambda_V$, is semiglobal, meaning that it holds in a properly large domain $\Omega$ containing the origin. It is challenging to determine the exact domain where the graph representation holds due to its close relationship with the regularity of $V$. Specifically, if $V$ is not differentiable at $x_0$, then the graph representation may not hold at $x_0$ (cf \cite{cannarsa2021singularities}).
\end{remark}

To find trajectories of \eqref{e:Hamiltonian-flow}, we solve a two-point BVP in a small neighborhood of the origin. Specifically, following \cite{sakamoto2008}, we set
$
\mathcal T=\left(
      \begin{array}{cc}
        I & S \\
        P & PS+I \\
      \end{array}
    \right),
$ and
$
\left(
  \begin{array}{c}
    \bar x \\
   \bar p \\
  \end{array}
\right)=\mathcal T^{-1}\left(
                    \begin{array}{c}
                      x \\
                      p \\
                    \end{array}
                  \right),\notag
$
where $S$ is the solution of the Lyapunov equation $(A-R(0)P)S+S(A-R(0)P)=R(0)$.
The system \eqref{e:Hamiltonian-flow} becomes
\begin{eqnarray}\label{e:Hamiltonian-flow-2}
\left\{\begin{array}{l}
         \dot{\bar x}=(A-R(0)P)\bar x+N_s(\bar x,\bar p) \\
         \dot{\bar p}=-(A-R(0)P)^T\bar p+N_{u}(\bar x,\bar p).
       \end{array}
\right.
\end{eqnarray}
Here $N_s(\bar x,\bar p), N_u(\bar x,\bar p)$ are the nonlinear terms except the linear terms in \eqref{e:Hamiltonian-flow}.
We solve \eqref{e:Hamiltonian-flow-2} with boundary condition
\begin{eqnarray}\label{e:bvp-n}
\bar x(0)=\bar x_0,\quad \bar p(+\infty)=0.
\end{eqnarray}
From \cite[Theorem 5]{sakamoto2008}, we have
\begin{prop}
Suppose conditions $(C_1)$-$(C_2)$ hold. For $\bar x_0$ sufficiently small, there exists unique trajectory of the two-point BVP \eqref{e:Hamiltonian-flow-2}-\eqref{e:bvp-n} contained in the stable manifold near $0$.
\end{prop}

\section{Asymptotic analysis of approximate closed-loop system}\label{s:errors}

This section provides an asymptotic analysis of the closed-loop system based on approximations of the stable manifold.
Hereafter, we assume that $f$, $q$, and $R$ satisfy conditions $(C_1)$-$(C_2)$, and that $V$ is the stabilizing solution of \eqref{e:HJB}.
In this paper, we use certain deep NN $p^{NN}(\theta,\cdot)$ to approximate $p(x)=\nabla V(x)$. The universal approximation theorem for NN guarantees that there exists NN approximation $p^{NN}(\theta,\cdot)$ that is sufficiently close to $p(x)$ if the number of parameters $\theta$ of the NN is large enough (see \cite{hornik1991approximation}, \cite{poggio2017}). The details of the NN will be presented in Section \ref{e:algorithm} below.
Recall that the exact optimal control is given by
\begin{eqnarray}\label{e:u*1}
u(x)=-W^{-1}g(x)^T p(x).
\end{eqnarray}
Subsequently, the approximation of the optimal control is
\begin{eqnarray}\label{e:uNN}
u^{NN}(x)=-W^{-1}g(x)^T p^{NN}(\theta, x).
\end{eqnarray}
Assume that $x(t)$ is the solution of
\begin{eqnarray}\label{e:x-closed}
\dot{x}=f(x)-R(x)p(x),\quad x(0)=x_0,
\end{eqnarray}
$x^{NN}(t)$ is the solution of the approximate closed-loop system
\begin{eqnarray}\label{e:x-closed-NN}
\dot{x}=f(x)-R(x)p^{NN}(\theta,x),\quad x(0)=x_0.
\end{eqnarray}
From Condition $(C_1)$, it holds that there exist constants $b>0$ and $k>0$ such that for all $|x|<b$,
\begin{eqnarray}\label{e:f-g-R}
&&|f(x)-Ax|<k|x|^2,\quad |g(x)-g(0)|<k|x|,\notag\\
&&|R(x)-R(0)|<k|x|.
\end{eqnarray}
Since $f,g,R$ are $C^{\infty}$ in $\Omega$, we have that $f,g,R$ are Lipschitz continuous in any compact subset of $\Omega$.
Moreover, Remark \ref{r:regularity} yields that $p$ is Lipschitz continuous in $\Omega$. Hereafter, without loss of generality, the Lipschitz constant of $f,g,R,p$ can be chosen as a unified constant $L$.

\begin{theorem}\label{t:x-decay}
Let $x(t)$ be the solution of \eqref{e:x-closed}. There exist constants $C>0$, $b>0$, $\alpha>0$ such that
$
|x(t)|\le C|x_0|e^{-\alpha t},
$
for $t\ge 0$ and $|x_0|\le b$.
\end{theorem}
\begin{proof}
From Theorem \ref{t:stable-manifolds} and \eqref{e:f-g-R}, there exist $k_0>0$ and $b>0$ such that
$
|f(x)-R(x)p(x)-(Ax-R(0)Px)|\le k_0|x|^2$,
for all $|x|<b.
$
Since all eigenvalues of $A-R(0)P$ has negative real part, the result holds from \cite[Theorem 2.77]{chicone2006ordinary}.
\end{proof}

\begin{remark}\label{r:decay}
Let $B=A-R(0)P$. There exists a constant $\beta>0$ such that
$
|e^{Bt}x_0|\le |x_0|e^{-\beta t}.
$
Hence $\alpha\to \beta^-$ as $b\to 0$ in Theorem \ref{t:x-decay}.
\end{remark}
\begin{cor}\label{e:x-decay-T}
Let $\Psi\subset\Omega$ be any compact set. For any $\varepsilon>0$, by Theorem \ref{t:x-decay}, there exists a constant $T_{\varepsilon}>0$ such that
$
|x(t)|<\varepsilon/2,\,\forall t\ge T_{\varepsilon},
$ for all $x_0\in \Psi$.
We call $T_{\varepsilon}$ an \emph{admissible time} of $x$ with respect to $\varepsilon$.
\end{cor}
\begin{proof}
From Theorem \ref{t:x-decay} and Remark \ref{r:regularity}, it holds that for any $x_0\in \Psi$, there exists $T_{\varepsilon, x_0}>0$ depending on $\varepsilon$ and $x_0$ such that
$
|x(t)|<\varepsilon/2,\,\forall t\ge T_{\varepsilon,x_0}.
$
Furthermore, using the continuous dependence of solutions on initial data of ODE and the compactness of $\Psi$, we have the existence of uniform constant $T_{\varepsilon}>0$ independent on $x_0$ such that $
|x(t)|<\varepsilon/2,\,\forall t\ge T_{\varepsilon}.
$
\end{proof}
For $x^{NN}(t)$, we have the following asymptotic result whose proof will be included in Appendix below.
\begin{theorem}\label{t:decay-xnn}
Let $P$ be the stabilizing solution of the Riccati equation \eqref{e:riccati}, let $\gamma_0>0$ be a fixed constant, and let $\Psi\subset\Omega$ be any compact set. For any $\varepsilon\in (0,\gamma_0)$, there exist $\delta>0$ and $\eta>0$ such that if $p^{NN}(\theta,\cdot)$ satisfies the following conditions:
(a) $|p^{NN}(\theta,x)-p(x)|<\delta$ for all $x\in \Omega$;
(b) $|p^{NN}(\theta,x)-Px|\le \eta|x|$ for $|x|<\gamma_0$,
then, $|x^{NN}(t)-x(t)|<\varepsilon$ for all $t>0$, and $x^{NN}(t)$ (resp. $x(t)$) decays exponentially as $t\to+\infty$. Here, $x^{NN}(t)$ is the solution of \eqref{e:x-closed-NN} (resp. the solution of \eqref{e:x-closed}) with arbitrary $x_0\in \Psi$. Consequently, $|J(x^{NN},u^{NN})-J(x,u)|<C\varepsilon$ where $C$ is a constant depending only on $f,g,R,p, W$.
\end{theorem}


\section{Algorithm}\label{e:algorithm}
In this section, we propose a deep learning algorithm based on the theoretical result in Theorem \ref{t:decay-xnn} to find $p^{NN}(\theta,\cdot)$ that is sufficiently close to $p(\cdot)$ in the sense of conditions (a)-(b) in Theorem \ref{t:decay-xnn}.

\subsection{Architecture of the NN}
To achieve certain accuracy, the number of parameters of the NN depends essentially on its architecture. For example, the complexity of deep compositional networks utilizing standard ReLUs for approximating Lipschitz continuous functions with a given accuracy depends linearly on the dimension $n$.  See e.g. \cite[Theorem 4]{poggio2017}. In contrast, the complexity required by an NN with only one hidden layer to achieve a similar level of accuracy grows exponentially with the dimension $n$. This phenomenon is known as the curse of dimensionality for the complexity of the NN.  See e.g. \cite[Theorem 1]{poggio2017}. In this paper, rather than delving further into the architecture of the NN, we \emph{empirically} use a deep NN of LSTM type with smooth activation function $\sin(\cdot)$ (see e.g. \cite{hochreiter1997long},\cite{sirignano2018dgm}). However, we should emphasize that other types of deep NN may also work well. As in \cite{sirignano2018dgm, nakamura2019adaptive, Onken2022Neural}, to mitigate the curse of dimensionality, we focus the investigation on control aspect and develop the algorithm based on gridfree method.


In the following, we utilize an LSTM type NN, $p^{NN}_o(\theta; x)$, as in \cite[Section 4.2]{sirignano2018dgm}.  Suppose input $x$ and output $p$ are $n$-dimensional, the number of hidden layers is $l$ and each layer has $m$ units. Different from \cite[Section 4.2]{sirignano2018dgm}, we use $\sigma:\mathbb R^m\to \mathbb R^m$, $\sigma(z)=(\sin(z_1),\cdots,\sin(z_m))$ as the activation function.
The derivative of $\sin(\cdot)$, namely $\cos(\cdot)$, has more global support than the derivatives of the traditional activation functions such as sigmoid.
Moreover, we modify the original NN by
$
p^{NN}(\theta,x)=p^{NN}_o(\theta,x)-p^{NN}_o(\theta,0).
$
Then $p^{NN}(\theta,0)=0$, which is a necessary condition for $p^{NN}$ to satisfy condition (b) in Theorem \ref{t:decay-xnn}.

To fit our theoretical conclusion in Section \ref{s:errors}, we define the loss function as follows: for $\nu\in [1,\infty]$,
\begin{eqnarray}\label{e:loss}
&&\mathcal L^\nu(\theta;\mathcal D)
:=\sigma_1\left[\frac{1}{|\mathcal D|}\sum_{i=1}^{|\mathcal D|}\|p_i-p^{NN}(\theta; x_i)\|^\nu\right]\\
&&~~+\sigma_2\max_{p_i\in \mathcal D}|p_i-p^{NN}(\theta; x_i)|
+\sigma_3\left\|\frac{\partial p^{NN}}{\partial x}(\theta,0)-P\right\|,\notag
\end{eqnarray}
where $|\cdot|$ denotes the standard Euclidean norm in $\mathbb R^n$, $\|\cdot\|$ is the operator norm of the matrix, $|\mathcal D|$ is the number of samples in $\mathcal D$, and $\sigma_i>0$ for $i=1,2,3$ are weight constants.
The first term measures the mean error with exponent $\nu$ between the NN predicted value $p^{NN}(\theta; x_i)$ and the standard value $p_i$ on dataset $\mathcal D$; the second term enforces that the maximum difference between the predicted and observed values is small; and the third term ensures that the derivative of the NN function at $0$ is close to the stabilizing solution $P$ of the Riccati equation \eqref{e:riccati}. The weight constants $\sigma_i$ control the relative importance of these terms in the overall loss function.


\subsection{The algorithm}\label{s:algorithm}
Inspired by \cite{nakamura2019adaptive}, we propose a deep learning algorithm based on adaptive data generation on the stable manifold. The sketch of the algorithm is as follows.

\emph{Step 0. Transformation of the model.}  We begin by rescaling the characteristic Hamiltonian system \eqref{e:Hamiltonian-flow}. The necessity of such a transformation is explained in Subsection \ref{s:rescaling} below.

\emph{Step 1. First generation of trajectories and sampling.}
We find a certain number of trajectories on the stable manifold by solving two-point BVP near the equilibrium and extending the local trajectories by IVP as described in Subsection \ref{s:bvp} below. Then we pick out some samples on each trajectory to obtain a training set $
\mathcal D_1:=\{(x_i,p_i)\}_{i=1}^{N_1}
$ as in Subsection \ref{s:exp} below.
Similarly, we also generate a set of samples, $\mathcal D^{\rm val}$, for validation of the trained NN later.

\emph{Step 2. First NN training.} We train an LSTM type NN, $p^{NN}(\theta,\cdot)$, with parameters $\theta$, on $\mathcal D_1$  satisfying
$
p^{NN}(\theta, x_i)\approx p_i,\, i=1,\cdots, N_1,
$
for some epochs.
If $p^{NN}(\theta,\cdot)$ satisfies ${\mathcal L^{\nu}}(\theta,\mathcal D_1)<\varepsilon$ for empirically chosen $\varepsilon$, then we check $p^{NN}(\theta,\cdot)$ on $\mathcal D^{\rm val}$ to obtain the test error.

\emph{Step 3. Adaptive data generation.} If the test error is not good enough, then we generate more samples near the points with relatively large errors from the first round training. Adding these new samples to $\mathcal D_1$, we obtain a larger sample set $\mathcal D_2$. For details see Subsection \ref{s:ag} below.

\emph{Step 4. Model refinement.} Based on the NN obtained in Step 2, we continue training the NN on the updated data set $\mathcal D_2$ as in Step 2. The training procedure stops once the NN satisfies a Monte Carlo test as in Subsection \ref{s:Monte-Carlo} below.

\emph{Step 5. Approximate optimal feedback control.} From $p^{NN}(\theta,\cdot)$, we can get the approximate optimal feedback control $u^{NN}$ \eqref{e:uNN} and compute the closed-loop trajectories at certain initial conditions $x(0)=x_0$ by solving \eqref{e:x-closed-NN}.

\begin{remark}
To enhance the performance and to prove the convergence of the procedure in the sense of mean square error (MSE) while minimizing computational costs, we can utilize a sample size selection scheme based on the sample variances of the training sets $\mathcal D_i$, $i=1,2,\cdots$, following the approach proposed in \cite[Section 4.1]{nakamura2019adaptive}.
\end{remark}

\subsection{Key techniques of the algorithm}\label{s:tech}
As mentioned above, our approach is to find NN approximation of the stable manifold, which differs from the method in \cite{nakamura2019adaptive}. Moreover, to improve the effectiveness of our algorithm, we incorporate several useful techniques as follows.
\subsubsection{Generating trajectories by two-point BVP and IVP}\label{s:bvp}
We first use \\
`scipy.integrate.solve\_bvp' to solve the two-point BVP \eqref{e:Hamiltonian-flow-2}-\eqref{e:bvp-n} for $x_0\in \partial B_r(0)$, where $r>0$ is small, and $\partial B_r(0)$ denotes the sphere centered at $0$ with radius $r$. This BVP solver implements a 4th-order collocation method with control of residuals (see \cite{jones2001scipy}).
The BVP solver requires an initial mesh for time and an initial guess for the solution values at each mesh node. The iterative procedure of the solver may diverge if the initial guess is not well chosen. For problem \eqref{e:Hamiltonian-flow-2}-\eqref{e:bvp-n}, the boundary condition $x_0$ are chosen on a small ball $\partial B_r(0)$, and the initial guess is constantly $0$ at each mesh node.
To select an appropriate radius $r$, we employ a Monte Carlo method. Specifically, we first choose a small $r$, randomly take some points on the ball $\partial B_r(0)$, then solve \eqref{e:Hamiltonian-flow-2}-\eqref{e:bvp-n}. If the success rate of the BVP solver is $100\%$, then we choose a slightly larger $r$. This procedure continues until the success rate of the BVP solver is not $100\%$.

We then extend the local solutions, $(\tilde x(t), \tilde p(t))$, $t\ge 0$, obtained by the BVP solver. On some interval $(T_-,0]$, $T_-<0$, we solve the following IVP:
\newpage 

\begin{eqnarray}\label{e:Hamiltonian-flow-ivp}
&&\left\{\begin{array}{l}
         \dot{x}= f(x)-R(x)p,\\
         \dot{p}= -(\dfrac{\partial f(x)}{\partial x})^Tp+\dfrac{1}{2}\dfrac{\partial (p^TR(x)p)^T}{\partial x}-(\dfrac{\partial q}{\partial x})^T,
       \end{array}
\right.\notag\\
&& \mbox{ with }x(0)=\tilde x(0), p(0)=\tilde p(0).
\end{eqnarray}
To solve this IVP, we can use `scipy.integrate.solve\_ivp' with various methods (e.g., `RK45', `Radau', etc.), as well as symplectic methods, variational integrators. Comparing to the BVP solver, the IVP \eqref{e:Hamiltonian-flow-ivp} is usually much easier to solve. To approximate $T_-$, we first choose a value $T_-^0<0$ and check whether the IVP solver successfully computes the solution for all initial conditions $(\tilde x(0),\tilde p(0))$ on $(T_-^0,0]$. If the success rate is $100\%$, we then choose a slightly smaller value $T_-^1<T_-^0<0$, and repeat the procedure until the IVP solver fails at some initial conditions.

\subsubsection{Sampling along trajectories}\label{s:exp}
Since Theorem \ref{t:x-decay} and Remark \ref{r:decay} yield that the decay of the trajectory $x(t)$ is almost $e^{-\beta t}$ as $t\to +\infty$, we empirically select samples on the trajectories near the origin according to an exponential distribution
$
\rho(t)=\frac{1}{\lambda} e^{-\frac{t}{\lambda}}
$ for $t>0$ with $\lambda=\beta_{\min}^{-1}$. Here $\beta_{\min}$ is the distance between the set of eigenvalues and the imaginary axis. That is, if we take $t_0,t_1,t_2,\cdots$ from the exponential distribution, then $(x(t_0),p(t_0)),(x(t_1),p(t_1)),(x(t_2),p(t_2)),\cdots$ lie on the stable manifold.
Moreover, we select a certain number of samples on $(T_-,0]$ according to uniform distribution.

\subsubsection{Adaptive data generation}\label{s:ag}
After the first round training, we record the absolute errors of $p^{NN}(\theta,\cdot)$ on $\mathcal D_1$ as
$
|p^{NN}(\theta, x_i)- p_i|,\, i=1,\cdots, N_1.
$
We then select the largest $[\mu N_1]$ points on $\mathcal D_1$, where $[y]$ is the integer part of $y$ and $\mu\in (0,1)$ is chosen as in \cite[Section 4.1]{nakamura2019adaptive}. Denote the set of these samples by $\mathcal {\hat D}_1$. We randomly sample $J_1$ points $y_j\in \Omega$ ($j=1,\cdots, J_1$) around $x_i$ with $(x_i,p_i)\in \mathcal {\hat D}_1$ according to some Gaussian distribution, then find solutions of \eqref{e:Hamiltonian-flow-2}-\eqref{e:bvp-n} with $x(0)=y_j$ by BVP solver. Here, the initial guess of the solution is the trajectory, $(x_i(t),p_i(t))$, $t\in (T_-,\infty)$, where $(x_i,p_i)$ lies on. We choose $L_1$ samples on each new trajectory according to a certain Gaussian distribution for $t>0$. Adding these new samples in $\mathcal D_1$, we obtain a larger sample set $\mathcal D_2$.

\subsubsection{Monte Carlo test of the NN} \label{s:Monte-Carlo}
To obtain the error bounds as in Theorem \ref{t:decay-xnn}, we use a Monte Carlo test. Specifically, we randomly select a certain number of initial points according to uniform distribution in $\Omega$. Using the trained NN, $p^{NN}(\theta,\cdot)$, we generate feedback control $u^{NN}$ as \eqref{e:uNN}, then check if the closed-loop trajectories is stabilized at these initial points by solving \eqref{e:x-closed-NN}. If the trained NN works well, then the test error gives the error bounds in Theorem \ref{t:decay-xnn}.

\subsubsection{Rescaling} \label{s:rescaling}
For concrete applications, the original formulation of models may not be suitable for numerical methods. Hence, modifications, such as coordinate transformations, should be made at the beginning of the algorithm. In our algorithm, we use rescaling, which also appears in computational optimal control (see, e.g., \cite{ross2018scaling, ross2007low}).
In \eqref{e:Hamiltonian-flow}, we apply a rescaling
$
(\hat x_1,\hat x_2,\cdots,\hat x_n)=(\lambda_1\bar x_1,\lambda_2\bar x_2,\cdots,\lambda_n\bar x_n)
$
so that the variables $(\hat x_1,\hat x_2,\cdots,\hat x_n)$ (and their derivatives) have the same orders of magnitude. There are two advantages of rescaling. Firstly, after rescaling, the proportion of convergent rate by the BVP solver increases. As mentioned, we set constant on mesh nodes as initial guess of the solution in `scipy.integrate.solve\_bvp'. Hence if the value and derivative of the exact solution are too large, then the BVP solver may diverge. Secondly, rescaling enhances the effectiveness of NN training. The loss function \eqref{e:loss} indicates that when the magnitudes of $x_1,x_2,\cdots,x_n$ differ significantly, certain variables in \eqref{e:loss} may carry more weight, while others may have relatively smaller weights and could potentially be negligible. Figure \ref{f:closed-loop2} in Section \ref{s:pendulum} below shows essence of the rescaling.

\section{Application to the Reaction Wheel Pendulums}\label{s:pendulum}

The RWP is a mechanical system that consists of a physical pendulum with a rotating disk (see Figure \ref{f:pendulum}). Researchers have been interested in the swing up and stabilization of various pendulums for the past two decades (see, e.g., \cite{sakamoto2013case, horibe2017optimal, block2007reaction}).

\begin{wrapfigure}[10]{r}{8em} %
    \begin{center}
        \includegraphics[width=0.23\textwidth]{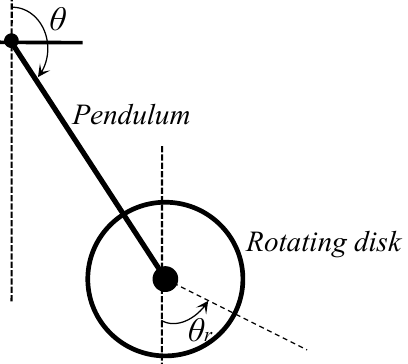}
        \vspace{-0.4cm}
        \caption{\footnotesize The Reaction Wheel Pendulum}
        \label{f:pendulum}
    \end{center}
\end{wrapfigure}
The ideal dynamical system of the RWP is given by
$\dot{x}_1= x_2,\,\dot{x}_2=a\sin x_1-b_pu,\,\dot{x}_3=b_ru.$
Here $u$ is an input function,
$x_1=\theta, x_2=\dot{\theta}, x_3=\dot{\theta}_r.$ we borrow the parameters of instrument in \cite[Page 21]{block2007reaction}:
$
a=78.4, \,b_p=1.08, \,b_r=198.
$
Define a rescaling by
$x_1= \lambda_1\bar x_1,\,x_2= \lambda_2\bar x_2,\,x_3=\lambda_3\bar x_3$.
Then the system becomes
$\dot{\bar x}_1=\frac{\lambda_2}{\lambda_1} \bar x_2,\,\dot{\bar x}_2=\frac{a}{\lambda_2}\sin (\lambda_1\bar x_1)-\frac{b_p}{\lambda_2}u,\,
    \dot{\bar x}_3=\frac{b_r}{\lambda_3}u.$
Let $k=\frac{\lambda_2}{\lambda_1}=\sqrt{a}$ and $\sigma=\frac{b_p}{\lambda_2}=\frac{b_r}{\lambda_3}$. We choose $\lambda_1=1$ since we are mainly concerned with $x_1\in [-\pi,\pi]$. Then $\lambda_2=\sqrt{a}$ and $\lambda_3=\frac{b_r}{b_p}\lambda_2$. With a little abuse of notations, we still use $(x_1,x_2,x_3)$ instead of $(\bar x_1,\bar x_2,\bar x_3)$ for simplicity.
Then $k=\lambda_2\approx8.85$, $\lambda_3\approx1623.30$, $\sigma=\frac{b_p}{\lambda_2}\approx 0.12$.
Letting $x=(x_1,x_2,x_3)^T$, $f(x)=(kx_2,k\sin x_1,0)^T$, $g(x)=(0,-\sigma,\sigma)^T$, the system has form \eqref{e:system}.
Define the instantaneous cost as
$
L(x,u)=\frac{1}{2}(x^Tx+0.01u^2).
$
We implement an LSTM type NN, $p^{NN}(\theta,\cdot)$, in PyTorch with $l=3$ and $m=50$.

We train the NN as the algorithm in Section \ref{s:algorithm}.

1) \emph{First sampling, training and validation:}
Set the error tolerance of \\
`scipy.integrate.solve\_bvp' to be $10^{-7}$. The infinite interval $[0,+\infty)$ can be replaced by $[0,20]$ to achieve numerical accuracy (since $\beta_{\min}=1.2$). The initial mesh of $t$ is
$
0, h, \cdots, 100 h,
$
with $h=0.2$, and the initial guess of the solution is $0$ at all these nodes.
We randomly select 200 points $x_i\in\partial B_{0.5}(0)$ according to the uniform distribution. With these points as boundary conditions $x_0$ for the BVP solver, we solves \eqref{e:Hamiltonian-flow-2}-\eqref{e:bvp-n} successfully.
Next, we solve the IVP \eqref{e:Hamiltonian-flow-ivp} using `scipy.integrate.solve\_ivp' with `method=Radau', `rtol=$10^{-5}$' (relative tolerance), and `atol=$10^{-7}$' (absolute tolerance). Monte Carlo test infers $T_-=-0.2$. Thus 200 trajectories are obtained on the stable manifold.
Finally, we select $5$ points in $[0,20]$ obeying the exponential distribution with $\lambda=0.83$, and choose $20$ points on $[-0.2,0]$ according to the uniform distribution. Then, we pick out the samples $(x_i,p_i)$ such that $x_i\in\Omega:=[-3.2,3.2]\times [-3.2,3.2]\times[-3.2,3.2]$. The set of these samples is denoted by $\mathcal D_1^{\rm train}$, with $|\mathcal D_1^{\rm train}|=4363$.
Similarly, solving $50$ trajectories on the stable manifold yields a validation set $\mathcal D^{\rm val}$ with size 1081.

We train the NN using the internal optimizer Adam in PyTorch with $6000$ epochs and learning rate $lr=0.001\times 0.5^{[j/1000]}$ where $j$ is the epoch. Set $\nu=1$, $\sigma_1=1.0$, $\sigma_2=\sigma_3=0.01$ in \eqref{e:loss}. After this round of training, the loss \eqref{e:loss} on $\mathcal D_1^{\rm train}$ is $7.5\times 10^{-3}$, and the test loss on $\mathcal D^{\rm val}$ is $1.09\times 10^{-2}$.
The running time of the training is about 250 seconds on a ThinkPad T480s laptop without GPU.

\begin{figure}[htbp]
\vspace{-0.35cm}
\begin{center}
\subfigure{
\includegraphics[width=0.40\textwidth]{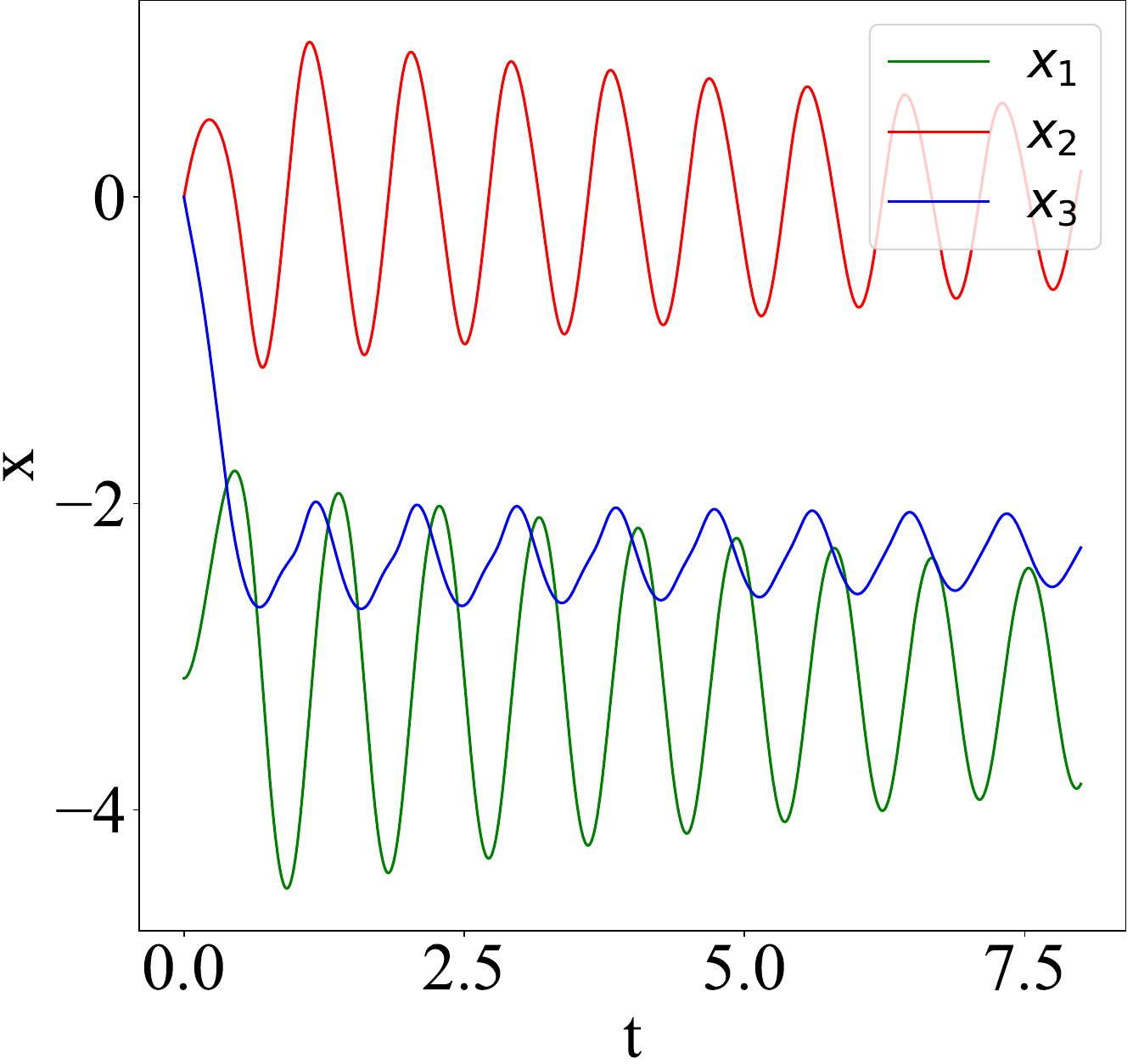}}
\subfigure{
\includegraphics[width=0.40\textwidth]{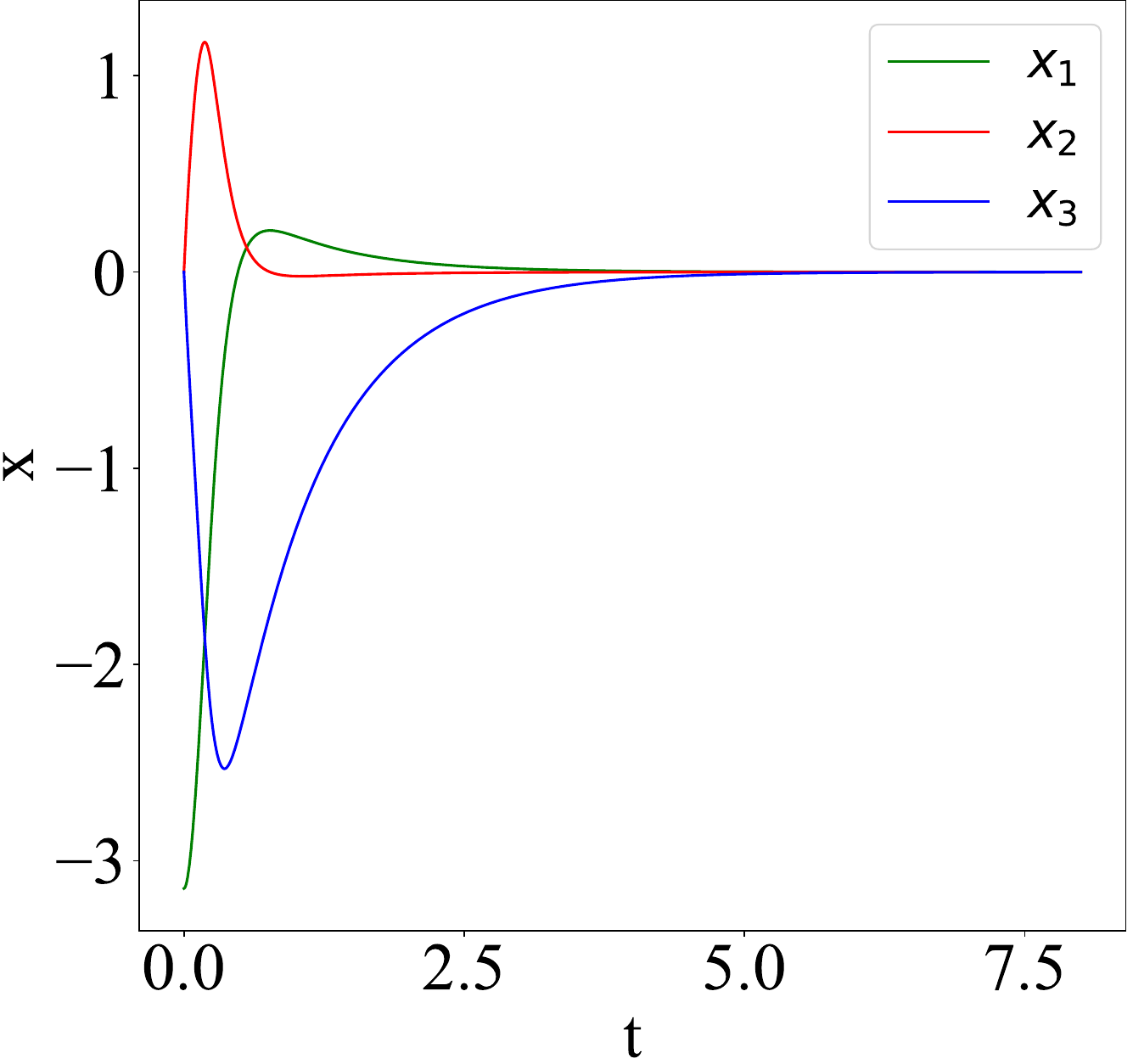}}
\end{center}
\hfill
\vspace{-0.8cm}
\caption{\footnotesize Simulations with initial states at hanging position based on first round training (left subfigure) and second round training (right subfigure).}
\label{f:round-1}
\end{figure}

2) \emph{Adaptive sampling and refinement of the NN:}
After the first round training, we find that the NN does not work well at some points, see e.g. Figure \ref{f:round-1}, by Monte Carlo test in Subsection \ref{s:Monte-Carlo}. To improve the NN, we adaptively generate samples as in Subsection \ref{s:ag} with $\sigma=0.1$, $J_1=5$ and $L_1=3$.
Then we add the new samples in $\mathcal D_1^{\rm train}$, obtaining updated training set $\mathcal D_2^{\rm train}$ with a size of 10652.
Based on the NN trained after the first round, we continue to train $p^{NN}(\theta,\cdot)$ on $\mathcal D_2^{\rm train}$ for 6000 epochs with learning rate $lr=2\times 10^{-4}\times 0.5^{[j/1500]}$, where $j$ denotes the epoch. The training time is approximately 1000 seconds. Training loss on $\mathcal D_2^{\rm train}$ is $7.3\times 10^{-3}$, and the test loss is $9.3\times 10^{-3}$. Monte Carlo test in Subsection \ref{s:Monte-Carlo} shows the NN works well.

3) \emph{Simulations:}
Numerical simulations are based on the NN generated feedback control \eqref{e:uNN}. Figure \ref{f:closed-loop-sim} shows the closed-loop trajectories at some states from \eqref{e:x-closed-NN}. Moreover, we compare the cost of the stable manifold (SM) method with the classical LQR and the optimal control obtained from the BVP solver. Table \ref{t:comparison} shows the costs at several points. It is evident that the cost of the stable manifold method is much smaller than that of LQR and is very close to the optimal control obtained from the BVP solver when $|x_0|$ is relatively large, whereas the difference between the three costs is small when $x_0$ is near the origin. Note that here we obtain the optimal control by using $(x^{NN}(t),p^{NN}(t))$ generated by the NN as the initial guess to solve the BVP \eqref{e:Hamiltonian-flow-2}-\eqref{e:bvp-n} successfully.

\begin{figure}[t]
\begin{center}
\subfigure[$x_0=(-\pi,0.5,0.3)$]{
\includegraphics[width=0.4\textwidth]{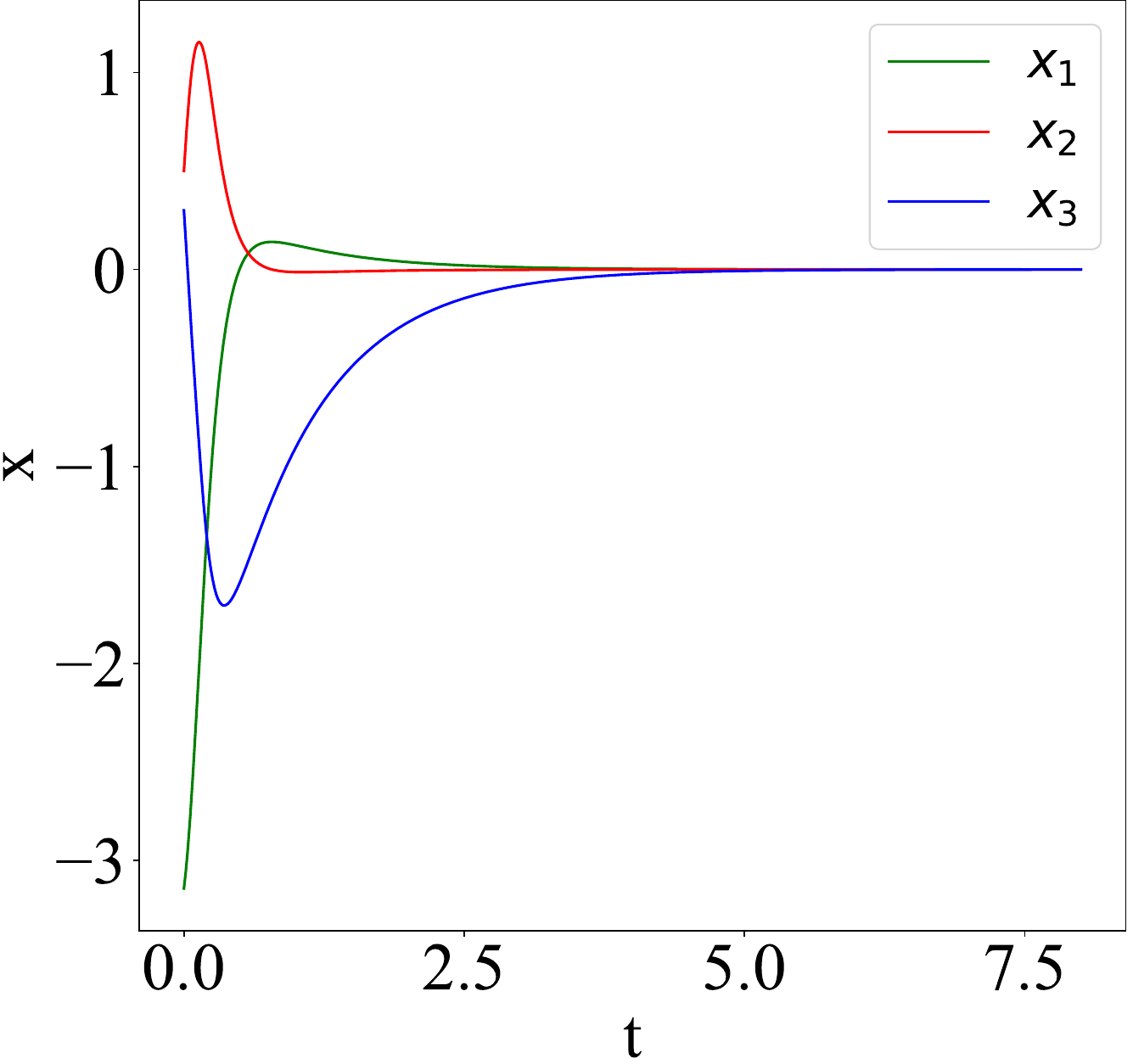}}
\subfigure[$x_0=(-0.3\pi,-0.78,0.26)$]{
\includegraphics[width=0.4\textwidth]{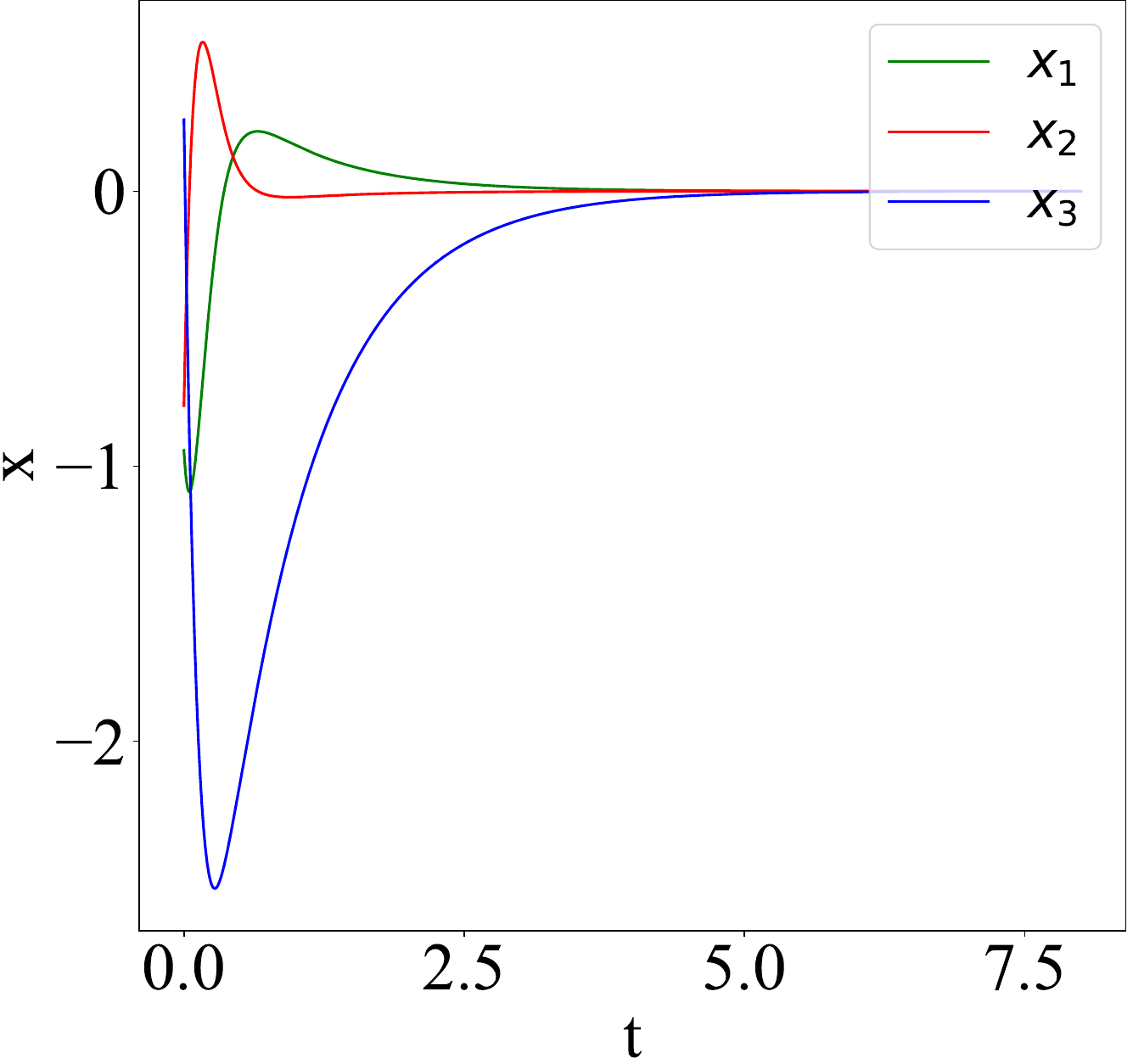}}
\end{center}
\hfill
\vspace{-0.5cm}
\caption{\footnotesize Closed-loop stabilizing trajectories of the Reaction Wheel Pendulum with some initial positions $x_0$ after the second round training.}
 \label{f:closed-loop-sim}
\end{figure}




\begin{table}[htbp]
\centering
\vspace{-0.2cm}
\caption{Comparison of the costs at certain points}
\vspace{0.0cm}
\begin{tabular}{cccc}
\hline
Initial positions & SM method & LQR & Optimal control\\
\hline
$(\pi,0,0)$ & 12.4 & 38.0 & 11.7\\
$(-0.8\pi,0.1,0.2)$ & 14.2 & 25.0 &14.0\\
$(-0.6\pi,-0.2,-0.4)$ & 20.5 & 25.5 & 20.4\\
\hline
$(0.1\pi,0,0)$ & 0.8096 & 0.8115 & 0.8096\\
$(0.2,0.03,-0.08)$ & 0.4055 & 0.4060 &0.4055\\
\hline
\end{tabular}\label{t:comparison}
\vspace{-0.2cm}
\end{table}
Additionally, Figure \ref{f:closed-loop2} shows that the system without rescaling at the beginning cannot generate good NN feedback controllers.
\begin{figure}[htbp]
\begin{center}
\subfigure{
\includegraphics[width=0.7\textwidth]{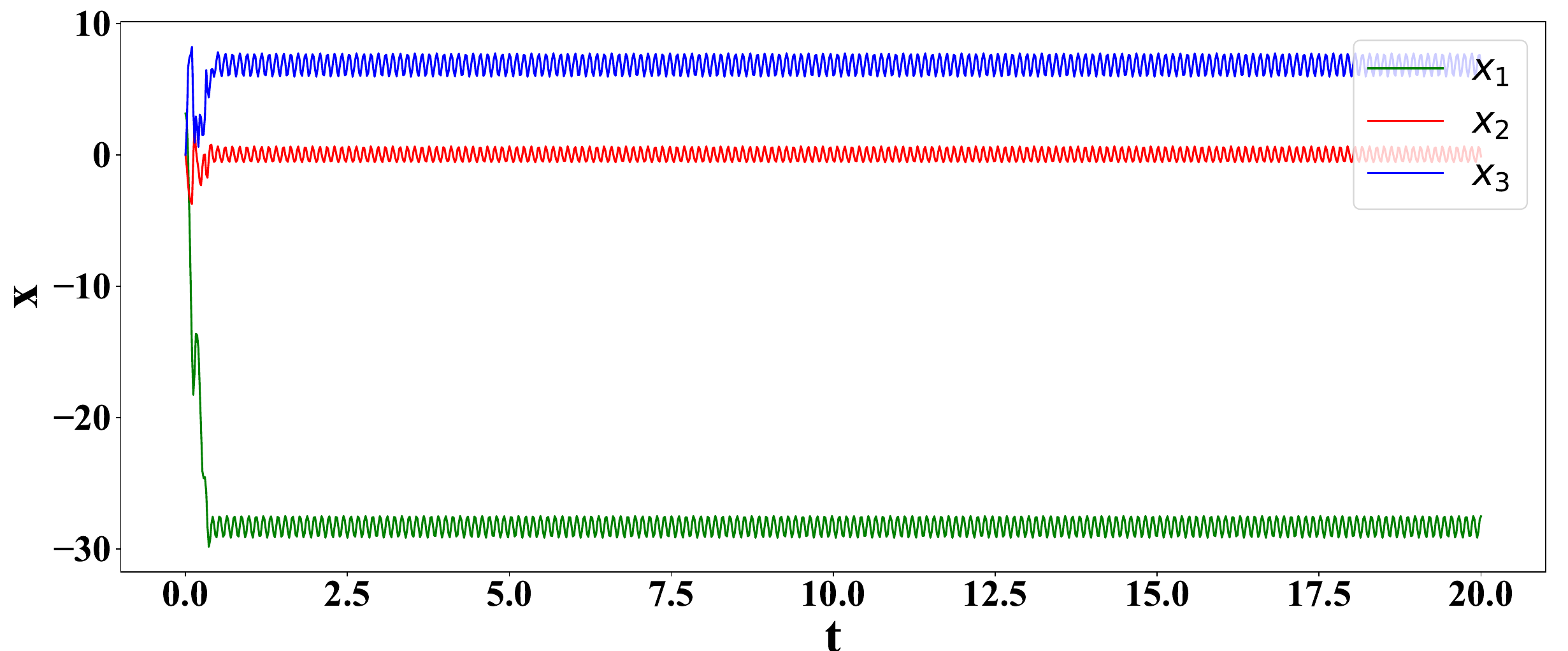}}
\end{center}
\hfill
\vspace{-0.8cm}
\caption{\footnotesize Sample closed-loop trajectory at $x_0=(\pi, 0, 0)$ after the second round of training without rescaling at the beginning. }
 \label{f:closed-loop2}
 \vspace{-0.3cm}
\end{figure}

\section{Application to optimal control for the parabolic Allen-Cahn equation}\label{s:AC}
To illustrate the effectiveness of our algorithm for high-dimensional problems, we give an application to optimal control of the parabolic Allen-Cahn (AC) equation as follows,
\begin{eqnarray}
&&\partial_t \mathcal X(\gamma,t)=\sigma \partial_{\gamma\gamma}\mathcal X(\gamma,t)+\mathcal X(\gamma,t)-\mathcal X(\gamma,t)^3+u,\notag\\
&&~~~~~~~~~~~~~~~~~~~~~~~~~~~~~~~~~~~~~~~~\mbox{ in }\mathcal I\times \mathbb R^+,\label{e:parabolic1}\\
&&\mathcal X(-1,t)=0,\quad  \mathcal X(1,t)=0, \quad t\in \mathbb R^+,\label{e:parabolic2}\\
&&\mathcal X(\gamma,0)=\mathcal X_0,\quad \gamma\in \mathcal I.\label{e:parabolic3}
\end{eqnarray}
Here $\mathcal I=[-1,1]$. The cost function is
$$
J(u, \mathcal X_0):=\frac{1}{2}\int_{0}^\infty \left(\|\mathcal X(\cdot, t)\|^2_{L^2(\mathcal I)}+\|u(t)\|_{L^2(\mathcal I)}^2\right)dt.
$$
The AC equation models a phase separation process (\cite{du2020phase}). The optimal control of AC equation is a typical infinite dimensional control problem (\cite{colli2015optimal, chrysafinos2022analysis}). In this example, we approximate the AC equation by a high-dimensional system.
Let $N$ be an integer greater than $3$, $h=\frac{2}{N}$ and
$
\gamma_i=-1+\frac{2i}{N},\, i=0,1,\cdots, N.
$
Then
$
\mathcal X(\gamma_0, t)=\mathcal X(\gamma_N,t)=0,
$
and for $i=1,\cdots, N-1$,
$
\mathcal X_{\gamma\gamma}(\gamma_i)\approx\frac{1}{h^2}(\mathcal X(\gamma_{i+1})-2\mathcal X(\gamma_i)+\mathcal X(\gamma_{i-1})).
$
Let
$X(t)=(X_1(t),\cdots, X_{N-1}(t))=(\mathcal X(\gamma_1,t),\cdots, \mathcal X(\gamma_{N-1},t)).$
Define
\begin{eqnarray}
A=\frac{1}{h^2}\left(
     \begin{array}{ccccccc}
       -2 & 1 & 0 & 0 &\cdots & 0 & 0 \\
       1 & -2 & 1 & 0 &\cdots & 0 & 0 \\
       \vdots & \vdots & \vdots & \vdots &\vdots&\vdots&\vdots \\
       0 & 0 & 0 & 0 &\cdots & 1 & -2 \\
     \end{array}
   \right)\notag
\end{eqnarray}
Hence $\mathcal X_{\gamma\gamma}\approx AX$. Set $u=(u_1,u_2,\cdots,u_{N-1})^T$.
Then \eqref{e:parabolic1}-\eqref{e:parabolic3} becomes a discrete control system
$
\frac{d}{dt}X=f(X)+u,$ where $X(0)=X_0=(\mathcal X_0(\gamma_1),\, \cdots, \,\mathcal X_0(\gamma_{N-1})),
$
$
f(X)=(\sigma A+I_{N-1})X-X^3
$, $X^3=(X_1^3,\cdots, X_{N-1}^3)$.
The cost function becomes
$\int_0^{\infty}\left(\frac{1}{N}\|X(t)\|^2+\frac{1}{N}\|u(t)\|^2\right)dt,$
where $\|\cdot\|$ denotes the standard Euclid norm in $\mathbb R^{N-1}$.
The HJB equation is
$$
\nabla V^T(X)f(X)-\frac{1}{2}\nabla V^T(X)R\nabla V(X)+\frac{1}{N}X^TX=0.
$$
Here
$
R:=\frac{N}{2}I_{N-1},
$
where $I_{N-1}$ is identity matrix of order $N-1$.
The feedback control
$u(X)=-\frac{N}{2}\nabla V(X).$
The Hamiltonian matrix at $0$ is hyperbolic for $N>3$.

In the numerical experiment, we choose $N=31$. We utilize an LSTM type NN via PyTorch with $l=3$ and $m=60$.
Let $\nu=1$, $\sigma_1=1$ and $\sigma_2=\sigma_3=0.1$ in the loss function \eqref{e:loss}.
The error tolerance of scipy.integrate.solve\_bvp is set to be $10^{-7}$, $[0,+\infty)$ is replaced by $[0,30]$ to ensure numerical accuracy since $\beta_{\min}=1.021$. The initial mesh of $t$ is
$
0, h, 2h, \cdots, 100 h,
$
with $h=0.3$.
We choose $B_{0.8}(0)\subset \mathbb R^{30}$ by Monte Carlo test, and select $1500$ points $x_i$ on the sphere $\partial B_{0.8}$ according to uniform distribution. Here the norm in $\mathbb R^{30}$ is
$
|X|=\sqrt{\frac{1}{30}\sum_{i=1}^{30}X_i^2}.
$
The BVP solver with these settings successfully solves \eqref{e:Hamiltonian-flow-2}-\eqref{e:bvp-n} for all $1500$ boundary conditions.
Next, we solve IVP \eqref{e:Hamiltonian-flow-ivp} using the `scipy.integrate.solve\_ivp' with `method=Radau, rtol=$10^{-5}$ (relative tolerance), atol=$10^{-7}$ (absolute tolerance)'. A Monte Carlo test shows $T_-=-0.03$.
Finally, we choose $23$ points in $[0,30]$
according the exponential distribution with $\lambda=1$ and select $3$ points according uniform distribution in $[-0.03,0]$. This yields a total of $39000$ samples, denoted by $\mathcal D_1^{\rm train}$.
Similarly, we generate a validation set $\mathcal D^{\rm val}$ by solving $200$ trajectories on the stable manifold.

We use the internal optimizer Adam in PyTorch with a learning rate $lr=0.01\times 0.5^{[j/1500]}$, where $j$ is the iterative times, and train the NN for $6000$ epochs. After training, we achieve the train loss on $\mathcal D_1^{\rm train}$ being $1.3\times 10^{-3}$ and the test loss on $\mathcal D^{\rm val}$ being $1.6\times 10^{-3}$.
The training process takes approximately $5400$ seconds on a ThinkPad T480s laptop without using GPU. Monte Carlo test as in Subsection \ref{s:Monte-Carlo} indicates that the trained NN works well.
Figure \ref{f:closed-loop6} demonstrates the trajectories at two initial states.


\begin{figure}[htbp]
\vspace{-0.3cm}
\begin{center}
\subfigure[$\mathcal X_0=2(x^2-1)$]{
\includegraphics[width=0.4\textwidth]{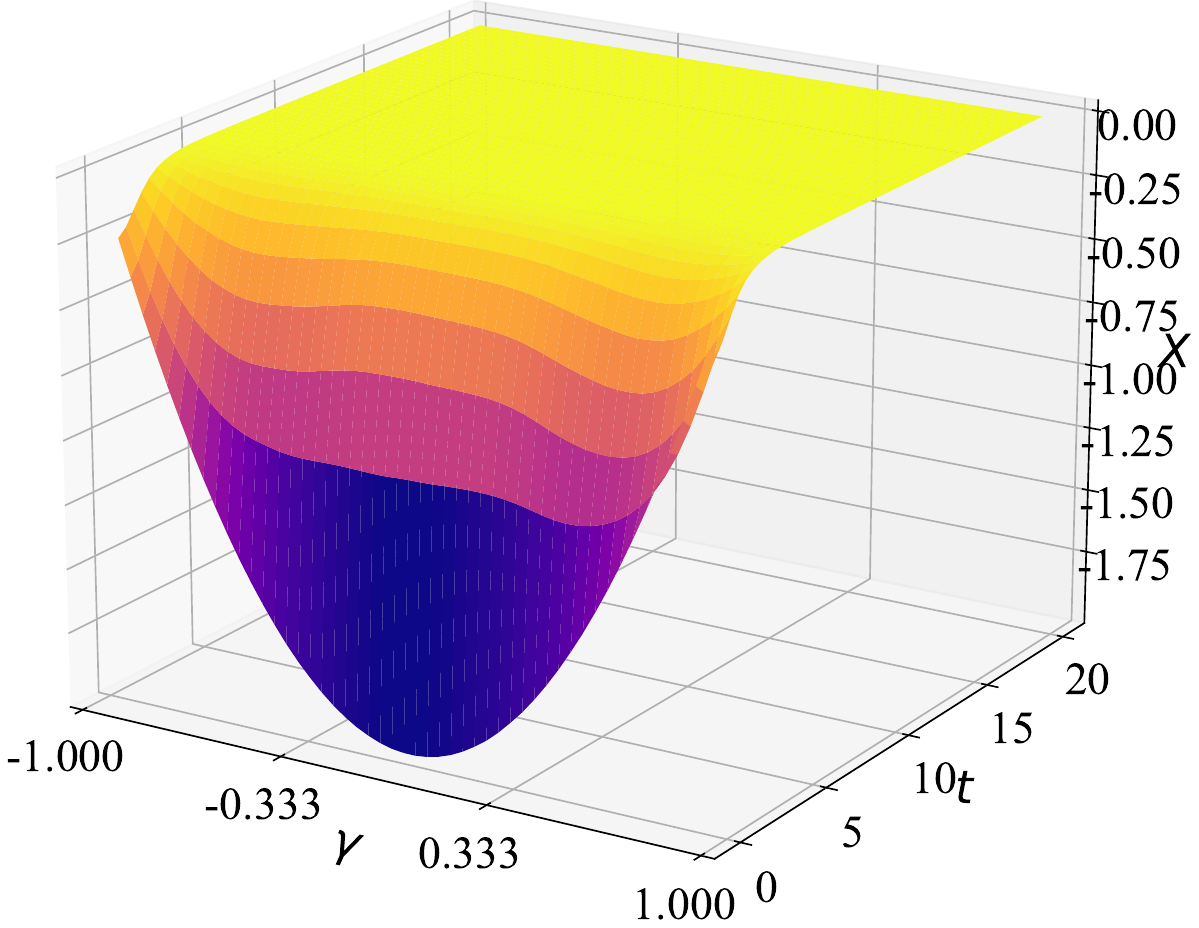}}
\subfigure[$\mathcal X_0=3.5\cos(1.5\pi x)$]{
\includegraphics[width=0.4\textwidth]{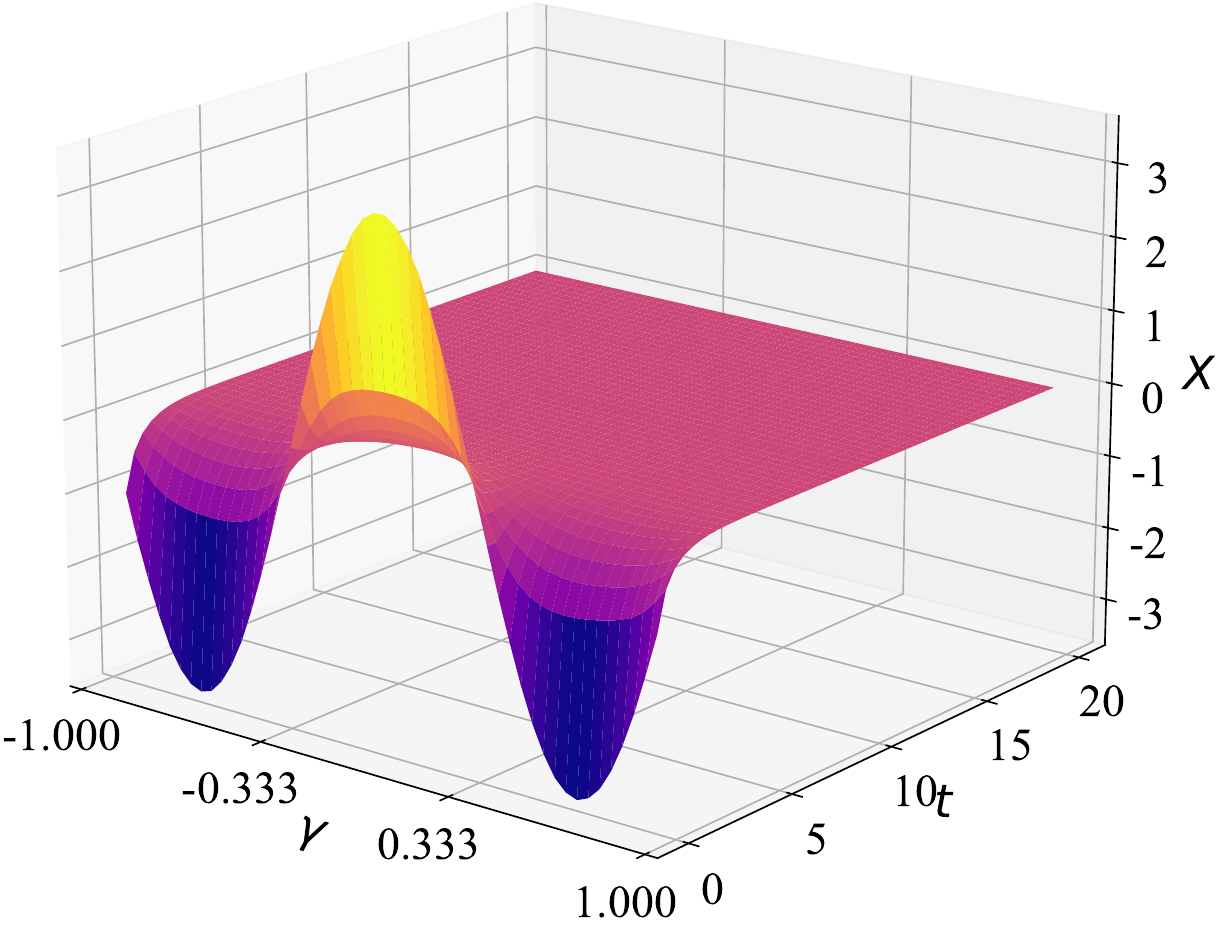}}
\end{center}
\hfill
\vspace{-0.5cm}
\caption{The dynamics of the NN controlled system at some initial states. }
 \label{f:closed-loop6}
 \vspace{-0.6cm}
\end{figure}


Finally, we compare the cost of the NN-generated controller, the LQR, and the standard BVP solver. When the initial function is large, the cost of the NN-generated feedback controller is much smaller than that of the LQR, and is very close to that of the standard BVP solver.
As in the previous example, we use $(x^{NN}(t),p^{NN}(t))$ generated by the trained NN as the initial guess to solve the BVP \eqref{e:Hamiltonian-flow-2}-\eqref{e:bvp-n}.
\begin{table}[htbp]
\vspace{-0.3cm}
\centering
\caption{Comparison of the costs for certain initial states}
\vspace{0.0cm}
\begin{tabular}{cccc}
\hline
Initial states & SM method & LQR & Optimal control\\
\hline
$3.5\cos(1.5\pi x)$ & 1.91 & 3.36 & 1.86\\
$2(x^2-1)$ & 2.04 & 2.92 & 1.96\\
$2.5(x-1)(x+1)^3$ & 2.15 & 4.39 & 2.10\\
\hline
$0.3\sin(\pi x)$ & 0.09673 & 0.09678 & 0.09667\\
$0.5(x-1)(x+1)^3$ & 0.0488 & 0.0513 & 0.0488\\
\hline
\end{tabular}\label{t:comparison2}
\vspace{-0.3cm}
\end{table}
\begin{remark}
Each evaluation of our NN-generated control signal takes on average less than one millisecond ($0.95\times 10^{-3}$ second) on a ThinkPad T480s laptop. This is faster than the controllers in \cite{Onken2022Neural} and \cite{nakamura2019adaptive}, whose time to generate control signals is at least several milliseconds. This fast control signal generation time is crucial for real-time applications.
\end{remark}

\section{Conclusion}

This paper proves that, under some natural conditions, NN approximations of the stable manifold of the HJB equation can generate nearly optimal controllers. Moreover, the approximate NN-controlled system is exponentially stable at the equilibrium as $t$ tends to $+\infty$.
Based on the theoretical conclusion, we propose an algorithm to construct a type of deep NN semiglobal approximations for the stable manifold. Our method relies on the geometric features of the HJB equations. The main advantage is that the derivatives of the value function of the optimal problem need not be calculated in the training procedure and computation of feedback control. The algorithm is based on adaptive data generation by finding trajectories on the stable manifold. Such kind of algorithm is grid-free, and is suitable for high-dimensional systems.
The effectiveness of our framework is illustrated by stabilizing the Reaction Wheel Pendulum and seeking optimal control of the parabolic Allen-Cahn equation.

\appendix

\section{Proof of Theorem \ref{t:decay-xnn}}\label{a:proof}
\begin{proof}[Proof of Theorem \ref{t:decay-xnn}]

Inspired by \cite[Proof of Theorem 2.77]{chicone2006ordinary}, we give a proof which focuses on the perturbation feature of the system.

Recalling that $f(x)=Ax+O(|x|^2)$, $R(x)=R(0)+O(|x|)$, $p(x)=Px+O(|x|^2)$ and $B=A-R(0)P$, we rewrite \eqref{e:x-closed} and \eqref{e:x-closed-NN} as
$$
\dot{x}=Bx+(f(x)-Ax)-(R(x)p(x)-R(0)Px)
:=Bx+n(x),\, x(0)=x_0,
$$
and
$$
\dot{x}=Bx+(f(x)-Ax)-(R(x)p^{NN}(\theta,x)-R(0)Px)
:=Bx+n^{NN}(x),\, x(0)=x_0.
$$

1. Note that
$$
x(t)=e^{Bt}x_0+\int_0^te^{B(t-s)}n(x(s))ds
$$
and
$$
x^{NN}(t)=e^{Bt}x_0+\int_0^te^{B(t-s)}n^{NN}(x^{NN}(s))ds.
$$
Then
$
x^{NN}(t)-x(t)
=\int_0^te^{B(t-s)}(n^{NN}(x^{NN}(s))-n(x(s)))ds.
$
We first assume $x^{NN}(t)\in \Omega$. Hence
$
|R(x^{NN}(t))|<C$, $|p(x^{NN}(t))|<C$,
$|R(x^{NN}(t))-R(x(t))|\le L|x^{NN}(t)-x(t)|.$
From Condition (a) in Theorem \ref{t:decay-xnn},
$
\left|n^{NN}(x^{NN}(s))-n(x(s))\right|
\le|f(x^{NN})-f(x)|+|R(x^{NN})|
|p^{NN}(\theta,x^{NN})-p(x^{NN})|
+|p(x^{NN})||R(x^{NN})-R(x)|
+|R(x)||p(x^{NN})-p(x)|
\le CL|x^{NN}-x|+C\delta.
$
Using Remark \ref{r:decay} and setting $y(t)=|x^{NN}(t)-x(t)|$, it holds that
\begin{eqnarray}\label{e:y}
y(t)
\le C\beta^{-1}\delta+\int_0^t  CLe^{-\beta(t-s)}y(s)ds.
\end{eqnarray}
Hence
$$
\frac{d}{dt}\log\left(C\beta^{-1}\delta+\int_0^t C Le^{-\beta(t-s)}y(s)ds\right)\le CL.\notag
$$
Therefore
$$\log\left(C\beta^{-1}\delta+\int_0^t CLe^{-\beta(t-s)}y(s)ds\right)
\le\log (C\beta^{-1}\delta)+CL t.$$
By \eqref{e:y},
$
y(t)\le C\beta^{-1}\delta e^{ CL t}.
$
Using Corollary \ref{e:x-decay-T}, it holds that for $$\delta<C^{-1}\beta e^{- CL T_{\varepsilon_1}}\varepsilon_1/2,$$
$
y(t)<\varepsilon_1/2,
$
$
|x^{NN}(T_{\varepsilon_1})|< \varepsilon_1$, and, $x^{NN}(t)\in \Omega$  for all $t\in [0,T_{\varepsilon_1}],
$
provided some $0<\varepsilon_1<\gamma_0$. Hence the assumption at the beginning of the proof is satisfied if $x(0)=x_0\in \Psi$ and ${\rm dist}(\partial\Psi,\partial \Omega)<\varepsilon_1$. Here ${\rm dist}(\partial \Psi,\partial \Omega)$ denotes the distance between the boundaries of $\Psi$ and $\Omega$.

2. Let $\varepsilon\in(0,\gamma_0)$ be a constant sufficiently small, and let $\varepsilon_1<\varepsilon$.
Assume $I=\{t\ge T_{\varepsilon_1}\,|\,|x^{NN}(t)|<\varepsilon\}$. Moreover, we have
$$
x^{NN}(t)=e^{Bt}x^{NN}(T_{\varepsilon_1})+\int_{T_{\varepsilon_1}}^te^{B(t-s)}n^{NN}(x^{NN}(s))ds.
$$
For $t\in I$, since $|p^{NN}(\theta,x^{NN}(t))-Px^{NN}(t)|\le \eta|x^{NN}(t)|$ (Condition (b) in Theorem \ref{t:decay-xnn}), it holds that
\begin{eqnarray}
&&|n^{NN}(x^{NN}(t))|\\
&\le&|f(x^{NN}(t))-Ax^{NN}(t)|
+|(R(x^{NN}(t))p^{NN}(\theta,x^{NN}(t))-R(0)Px^{NN}(t)|\notag\\
&\le& k|x^{NN}(t)|^2+|R(x^{NN}(t))-R(0)||Px^{NN}(t)|\notag\\
&&+|R(x^{NN}(t))||p^{NN}(\theta,x^{NN}(t))-Px^{NN}(t)|\notag\\
&\le& C_1((k+1)\varepsilon+\eta) |x^{NN}(t)|,\notag
\end{eqnarray}
where $C_1$ is a constant depending only on $f,R,p$.
It follows that
\begin{eqnarray}
|x^{NN}(t)|
&\le& C|x^{NN}(T_{\varepsilon_1})| e^{-\beta (t-T_{\varepsilon_1})}\label{e:xNN-g}\\
&+&CC_1((k+1)\varepsilon+\eta) \int_{T_{\varepsilon_1}}^te^{-\beta(t-s)} |x^{NN}(s)|ds.\notag
\end{eqnarray}
Rewrite \eqref{e:xNN-g} as
$$
|x^{NN}(t)|e^{\beta (t-T_{\varepsilon_1})}\le C|x^{NN}(T_{\varepsilon_1})|
+CC_1((k+1)\varepsilon+\eta) \int_{T_{\varepsilon_1}}^te^{\beta(s-T_{\varepsilon_1})} |x^{NN}(s)|ds.
$$
Using the Gronwall inequality (see e.g. \cite[Theorem 2.1]{chicone2006ordinary}), we find that
\begin{eqnarray}\label{e:xnn}
|x^{NN}(t)|
&\le &C|x^{NN}(T_{\varepsilon_1})|e^{-\alpha(t-T_{\varepsilon_1})},
\end{eqnarray}
where $\alpha=\beta-CC_1((k+1)\varepsilon+\eta)$. If the two positive constants $\varepsilon, \eta$ sufficiently small, then $\alpha>0$. In the following, we assume that the constant $C$ in \eqref{e:xnn} is greater than 2.

3. We prove that for $\varepsilon_1<\varepsilon/C$, $I=[T_{\varepsilon_1},\infty)$. If not, then there exists a $\bar T<\infty$ such that $I=[T_{\varepsilon_1},\bar T)$. Since $|x^{NN}(T_{\varepsilon_1})|<\varepsilon/C$, from \eqref{e:xnn} we get that
$
|x^{NN}(t)|< \varepsilon e^{-\alpha(t-T_{\varepsilon_1})}<\varepsilon,\, \forall t\in [T_{\varepsilon_1},\bar T).\notag
$
Then by the extension theorem of ODE, there is some small $\tau>0$ such that the solution is defined in the $[T_{\varepsilon_1},\bar T+\tau)$ and
$
|x^{NN}(t)|\le \varepsilon e^{-\alpha(t-T_{\varepsilon_1})}<\varepsilon,\,\forall t\in [T_{\varepsilon_1},\bar T+\tau).\notag
$
That is a contradiction by the definition of $I$.

4. In summary, let $C$ be fixed constant in \eqref{e:xnn} larger than $2$, and let $\varepsilon$ be a constant in $(0,\gamma_0)$ sufficiently small. For $0<\varepsilon_1<\varepsilon/C<\gamma_0$ and $$\delta<C^{-1}\beta e^{- CL T_{\varepsilon_1}}\varepsilon_1/2,$$ we have that
$$
|x^{NN}(t)-x(t)|\le \varepsilon,\quad \forall t\in(0,\infty),\notag
$$
and
$$
|x^{NN}(t)|<\varepsilon e^{-\alpha (t-T_{\varepsilon_1})}, \, \forall t\in (T_{\varepsilon_1},\infty).\notag
$$
This yields the first conclusion.

5. Further, we have
$
|q(x^{NN})-q(x)| \le 2 |Q| |x| |x^{NN}-x|\le C|x^{NN}-x|,
$
and
\begin{eqnarray}
&&|u^{NN}(x^{NN})-u(x)|\\
&\le& |W^{-1}|[(|g(x^{NN})-g(x)||p^{NN}(x^{NN})|+\notag\\
&&|g(x)||p^{NN}(x^{NN})-p(x^{NN})|
+|g(x)||p(x^{NN})-p(x)|)]\notag\\
&\le& C (|x^{NN}-x|+\delta),\notag
\end{eqnarray}
where $C>0$ is a constant depending only on $g,p,W$.
Hence by Theorem \ref{t:decay-xnn}, for $\delta>0$ sufficiently small, it holds that for some $\varepsilon<\gamma_0$,
$$
|J(x^{NN},u^{NN})-J(x,u)|\le C_2\varepsilon,
$$
where $C_2$ is a constant depending only on $f,g,R,p, W$. This completes the proof.
\end{proof}

\section*{Acknowledgements}

The author is greatly indebted to Prof. Wei Kang for many helpful discussions and suggestions. The author would like to express his appreciation to Prof. Qi Gong and Dr. Tenavi Nakamura-Zimmerer for the useful suggestions and comments on the paper. Moreover, the author deeply thanks the anonymous referees for their insightful comments and suggestions which essentially improve the quality of this paper.

\bibliographystyle{plain}

\end{document}